\documentclass[a4paper,12pt,reqno]{amsart}

\usepackage{amsmath}
\usepackage{amssymb}
\usepackage{amsfonts}
\usepackage{graphicx}
\usepackage{mathtools} 
\usepackage{bbm} 
\usepackage[colorlinks]{hyperref}
\newtheorem{theorem}{Theorem}[section]
\newtheorem{lemma}[theorem]{Lemma}

\newtheorem{corollary}[theorem]{Corollary}
\theoremstyle{remark}
\newtheorem{remark}[theorem]{Remark}
\newcommand{\R}{\mathbb{R}}
\newcommand{\Z}{\mathbb{Z}}
\newcommand{\C}{\mathbb{C}}
\newcommand{\dist}{\operatorname{dist}}
\newcommand{\ud}{\,\mathrm{d}}
\newcommand{\Real}{\operatorname{Re}}
\newcommand{\Imag}{\operatorname{Im}}
\newcommand{\ab}{A_\infty}

\DeclarePairedDelimiter\abs{\lvert}{\rvert}
\DeclarePairedDelimiter\norm{\lVert}{\rVert}

\title[Explicit constants in $L^p$-Hardy inequalities]{Explicit constants in $L^p$-Hardy inequalities for Aharonov--Bohm potentials}

\author[D. Suragan]{Durvudkhan Suragan}
\address{
	Durvudkhan Suragan:
	\endgraf
	Department of Mathematics
	\endgraf
	Nazarbayev University
	\endgraf
	Astana 010000
	\endgraf
	Kazakhstan
	\endgraf
	{\it E-mail address:} {\rm  durvudkhan.suragan@nu.edu.kz}
}


\begin{document}

\begin{abstract}
For the two-dimensional Aharonov--Bohm potential $A_\beta$ with flux $\beta\notin\Z$ and $1<p<2$, Cazacu, Krej\v{c}i\v{r}\'{\i}k, Lam and Laptev \cite{CKLL} proved by a compactness argument that their constant $\lambda_\beta(p)$ in the $L^p$-Hardy inequality strictly exceeds the free constant $\big(\tfrac{2-p}{p}\big)^p$, and asked for a constructive proof with explicit estimates and for comparability of $\lambda_\beta(p)$ with a quantity depending on $\dist(\beta,\Z)$. We answer both questions by using a compactness-free two-sided bound for the twisted angular constant.
Our explicit Hardy constant is $$\big[\big(\tfrac{2-p}{p}\big)^{2}+\big(\tfrac{\sin(\pi\dist(\beta,\Z))}{\pi}\big)^{2}\big]^{p/2},\quad 1<p<2.$$ 
As a byproduct we observe that when $p\ge 2$ the Aharonov--Bohm field produces an $L^p$-Hardy inequality with the usual homogeneous weight $\abs{x}^{-p}$. Our approach also provides new $L^p$-Hardy inequalities with explicit constants for the complex 
AB potentials.
\end{abstract}

\maketitle

\section{Introduction and main results}\label{sec:intro}

It is classical, going back to Hardy and Leray, that for $d\ge 2$ and $1<p<d$ one has the $L^p$-Hardy inequality
\begin{equation}\label{eq:freehardy}
\int_{\R^d}\abs{\nabla u}^{p}\ud x\ge\mu_{p,d}\int_{\R^d}\frac{\abs{u}^{p}}{\abs{x}^{p}}\ud x,
\end{equation}
for all $u\in W^{1,p}(\R^d)$, with the optimal constant $\mu_{p,d}:=\Big(\frac{d-p}{p}\Big)^{p}$. In the borderline and supcritical regimes $p\ge d$, by contrast, the free $p$-Laplacian becomes critical, that is, not only does \eqref{eq:freehardy} fail, but no inequality $\int_{\R^d}\abs{\nabla u}^{p}\ud x\ge\int_{\R^d}V\abs{u}^{p}\ud x$ over $u\in C_c^{\infty}(\R^d)$ survives with any nontrivial nonnegative potential $V$ in place of $\mu_{p,d}\abs{x}^{-p}$; see \cite[Proposition 1.1]{CKLL}.

A remarkable phenomenon, discovered in the linear case $p=2$ by Laptev and Weidl \cite{LW}, is that a magnetic field can restore (or even improve) Hardy inequalities. 

For the two-dimensional Aharonov--Bohm (AB) potential
\begin{equation}\label{eq:AB}
A_\beta(x)\;=\;\beta\,\frac{(x_2,-x_1)}{\abs{x}^{2}},\quad \beta\in\R,\quad x\in\R^{2}\setminus\{0\},
\end{equation}
whose magnetic field vanishes away from the origin but which carries the nontrivial circulation (flux) $-2\pi\beta$ around it, one has \cite{LW}
\begin{equation}\label{eq:LW}
\int_{\R^2}\abs{\nabla_{A_\beta}u}^{2}\ud x\ge\dist(\beta,\Z)^{2}\int_{\R^2}\frac{\abs{u}^{2}}{\abs{x}^{2}}\ud x,
\quad u\in C_c^{\infty}(\R^{2}\setminus\{0\}),
\end{equation}
with the sharp constant $\dist(\beta,\Z)^{2}$; here $\nabla_{A_\beta}u:=\nabla u+iA_\beta u$ denotes the AB magnetic gradient acting on complex-valued functions. Thus, although the free Hardy inequality fails entirely for $d=2=p$, a non-integer AB flux produces one, quantified exactly by the distance of the flux to the integers. For further developments in the $L^2$-setting we refer to \cite{BDELL,CK,FKLV,LY} and the references therein. 

In a complementary direction, Krej\v{c}i\v{r}\'{i}k \cite{Kre19} extended the Laptev--Weidl inequality \eqref{eq:LW} to complex-valued AB potentials, showing that even a purely imaginary flux restores the Hardy inequality. In the present paper, we also aim to obtain $L^p-$analogues of this complex AB potential result.

The extension of this circle of ideas to the nonlinear setting $p\neq 2$ was initiated by Cazacu, Krej\v{c}i\v{r}\'{\i}k, Lam and Laptev \cite{CKLL}. Among other results, they proved:

\begin{theorem}[{\cite[Theorem 1.4]{CKLL}}]\label{thm:CKLL14}
Let $1<p<2$ and let $A_\beta$ be given by \eqref{eq:AB}. If $\beta\notin\Z$, then there exists a constant
$\lambda_\beta(p)>\big(\tfrac{2-p}{p}\big)^{p}$
such that
\begin{equation}\label{eq:hardyAB}
\int_{\R^2}\abs{\nabla_{A_\beta}u}^{p}\ud x\ge\lambda_\beta(p)\int_{\R^2}\frac{\abs{u}^{p}}{\abs{x}^{p}}\ud x,
\quad\forall\,u\in C_c^{\infty}(\R^{2}).
\end{equation}
\end{theorem}

The proof of Theorem \ref{thm:CKLL14} in \cite{CKLL} reduces, via a reverse Minkowski inequality (recalled in Section \ref{sec:reduction} below), to the strict positivity of the twisted angular constant
\begin{equation}\label{eq:lambdadef}
\lambda(\beta,p)\;:=\;\inf_{\substack{u\in W^{1,p}(0,2\pi)\setminus\{0\}\\ u(0)=u(2\pi)}}
\frac{\displaystyle\int_{0}^{2\pi}\abs{\partial_\varphi u+i\beta u}^{p}\ud\varphi}
{\displaystyle\int_{0}^{2\pi}\abs{u}^{p}\ud\varphi}\,.
\end{equation}
In \cite{CKLL} the positivity $\lambda(\beta,p)>0$ for $\beta\notin\Z$ is established by contradiction, using weak compactness in $W^{1,p}(0,2\pi)$, the compact embedding into $L^p$, Mazur's lemma and the H\"older embedding $W^{1,p}\hookrightarrow C^{0,1-1/p}$. This argument is soft: it produces no numerical lower bound. Accordingly, \cite[Remark 1.2]{CKLL} poses the following open problem: give a constructive proof of Theorem \ref{thm:CKLL14} with explicit estimates on $\lambda_\beta(p)$, and decide whether $\lambda_\beta(p)$ is comparable with a quantity depending on $\dist(\beta,\Z)$, as in the sharp result \eqref{eq:LW} for $p=2$.

We remark that the companion open problem raised in \cite[Remark 1.1]{CKLL}, concerning the constant $c_p$ in the improved Hardy inequalities with general magnetic fields, has recently been resolved in \cite{CT}; the AB problem of \cite[Remark 1.2]{CKLL} addressed here is, to the best of our knowledge, untouched by these works.

In this note we resolve the problem of \cite[Remark 1.2]{CKLL}. Our main ingredient are the following quantitative estimates  on the angular constant \eqref{eq:lambdadef} for all $1<p<\infty$.

\begin{theorem}\label{thm:angular}
Let $1<p<\infty$, $\beta\in\R$, and set $\theta:=\dist(\beta,\Z)\in[0,\tfrac12]$. Then
\begin{equation}\label{eq:twosided}
\left(\frac{\sin(\pi\theta)}{\pi}\right)^{p}\le\lambda(\beta,p)\;\le\;\theta^{p}.
\end{equation}
In particular,
\begin{equation}\label{eq:comparable}
\Big(\tfrac{2}{\pi}\Big)^{p}\,\dist(\beta,\Z)^{p}\le\lambda(\beta,p)\le\dist(\beta,\Z)^{p}.
\end{equation}
\end{theorem}

At $p=2$, expansion into the Fourier basis $\{e^{in\varphi}\}_{n\in\Z}$ shows that the exponentials are exact minimisers and $\lambda(\beta,2)=\theta^{2}$; thus the upper bound in \eqref{eq:twosided} is attained at $p=2$, and the lower bound loses only the universal factor.
The proof of the lower bound in Theorem \ref{thm:angular} is simple and uses no compactness. 
Feeding Theorem \ref{thm:angular} into the scheme of the proof of \cite[Theorem 1.4]{CKLL} produces the explicit magnetic Hardy inequality in the range $1<p<2$.

\begin{theorem}\label{thm:main}
Let $1<p<2$, $\beta\in\R$, and $\theta:=\dist(\beta,\Z)$. Then for all $u\in C_c^{\infty}(\R^{2})$,
\begin{equation}\label{eq:main}
\int_{\R^2}\abs{\nabla_{A_\beta}u}^{p}\ud x
\ge
\Bigg[\Big(\frac{2-p}{p}\Big)^{2}+\Big(\frac{\sin(\pi\theta)}{\pi}\Big)^{2}\Bigg]^{p/2}
\int_{\R^2}\frac{\abs{u}^{p}}{\abs{x}^{p}}\ud x .
\end{equation}

\end{theorem}

According to \eqref{eq:main}, for $\forall\beta\notin\Z$, for the best/sharp constant $c_{best}$ (the optimal global constant) one may take 
$$
\Bigg[\Big(\frac{2-p}{p}\Big)^{2}+\dist(\beta,\Z)^{2}\Bigg]^{p/2} \ge c_{best}\ge\lambda_\beta(p)$$$$\ge
\Bigg[\Big(\frac{2-p}{p}\Big)^{2}+\Big(\frac{2\,\dist(\beta,\Z)}{\pi}\Big)^{2}\Bigg]^{p/2}.
$$
Here, the lower bound on $\lambda_\beta(p)$ 
follows from $\sin(\pi\theta)\ge2\theta$ and the upper bound follows from a standard logarithmic cutoff
$u_L(r,\varphi)
=
r^{-(2-p)/p}
\eta_L(\log r)e^{in\varphi},$
for some $n$ such that  $|n+\beta|=\theta$.

Because Theorem \ref{thm:angular} holds for all $1<p<\infty$, it also yields new information in the critical and supcritical regimes $p\ge 2$. Note that \cite[Theorem 1.1]{CKLL} does not apply here: it requires a smooth, 
nontrivial magnetic field, whereas the AB field vanishes identically on $\R^2\setminus\{0\}$.

\begin{theorem}\label{cor:allp}
Let $2\leq p<\infty$, $\beta\in\R$, and $\theta:=\dist(\beta,\Z)$. Then for all $u\in C_c^{\infty}(\R^{2}\setminus\{0\})$,
\begin{equation}\label{eq:allp}
\int_{\R^2}\abs{\nabla_{A_\beta}u}^{p}\ud x
\ge \left[
\left(\frac{p-2}{p}\right)^p+
\left(\frac{\sin(\pi\theta)}{\pi}\right)^p
\right]\int_{\R^2}\frac{\abs{u}^{p}}{\abs{x}^{p}}\ud x .
\end{equation}
\end{theorem}

At $p=2$ the constants of \eqref{eq:main} and \eqref{eq:allp} both become $\big(\tfrac{\sin(\pi\theta)}{\pi}\big)^2$, reproducing the sharp Laptev--Weidl constant $\theta^{2}$ up to the universal factor $\big(\tfrac{\sin(\pi\theta)}{\pi\theta}\big)^2\in[(2/\pi)^{2},1)$, which moreover tends to $1$ as $\theta\to0$  (the regime in which Theorem \ref{thm:CKLL14} is most delicate.) 
Note that the constant $\big(\sqrt{(2-p)^{2}+\beta^{2}p^{2}}/p\big)^{p}$ obtained  in \cite[Theorem 1.5]{CKLL} grows without bound in $\abs{\beta}$ and 
yields \eqref{eq:hardyAB} only for real-valued $u$. By contrast, \eqref{eq:main} holds for all complex-valued test functions, with a constant that is $1$-periodic in $\beta$. In physical terms, the $1$-periodicity in $\beta$ of all the constants appearing here expresses flux quantisation in the sense of Byers and Yang \cite{BY}.
 The obtained lower bounds may be also regarded as an $L^p$-analogue, in the continuum, of frustration-index lower bounds for eigenvalues of magnetic Laplacians on graphs (see e.g. \cite{LLPP}).

This note has a simple structure. Section \ref{sec:angular} contains the proof of Theorem \ref{thm:angular}. Section \ref{sec:reduction} presents the reduction of Theorem \ref{thm:main} to Theorem \ref{thm:angular} following the scheme of \cite{CKLL}, together with the short proof of Theorem \ref{cor:allp}. In Section \ref{sec:complexAB},  we extend our results to the complex AB potential.

\section{Proof of Theorem \ref{thm:angular}}\label{sec:angular}

Throughout this section $1<p<\infty$ and $\theta:=\dist(\beta,\Z)\in[0,\tfrac12]$. If $\beta\in\Z$ both sides of \eqref{eq:twosided} vanish, which is periodic since $\beta\in\Z$), so we assume $\beta\notin\Z$, i.e.\ $\theta\in(0,\tfrac12]$.

Let $u\in W^{1,p}(0,2\pi)$ with $u(0)=u(2\pi)$; the pointwise values are well defined by the embedding $W^{1,p}(0,2\pi)\hookrightarrow C[0,2\pi]$.

Define
\[
w(\varphi):=e^{i\beta\varphi}\,u(\varphi),\qquad \varphi\in[0,2\pi].
\]
Then $w\in W^{1,p}(0,2\pi)$, $\abs{w}=\abs{u}$ pointwise, and
\[
w'(\varphi)=e^{i\beta\varphi}\big(\partial_\varphi u+i\beta u\big)(\varphi),
\quad\text{so}\quad
\abs{w'}=\abs{\partial_\varphi u+i\beta u}\ \ \text{a.e.}
\]
Thus the twisted Rayleigh quotient of $u$ equals the free Rayleigh quotient of $w$, at the price that the periodicity of $u$ becomes the twisted boundary condition
\begin{equation}\label{eq:twist}
w(2\pi)=e^{2\pi i\beta}\,w(0).
\end{equation}
Since $u\mapsto w$ is a bijection between the two admissible classes,
\begin{equation}\label{eq:lambdatwist}
\lambda(\beta,p)=\inf\Big\{\tfrac{\int_0^{2\pi}\abs{w'}^{p}}{\int_0^{2\pi}\abs{w}^{p}} : w\in W^{1,p}(0,2\pi)\setminus\{0\}, w(2\pi)=e^{2\pi i\beta}w(0) \Big\}.
\end{equation}

Extend $w$ to $\R$ by
\[
w(\varphi+2\pi k):=e^{2\pi i\beta k}\,w(\varphi),\quad \varphi\in[0,2\pi),\ k\in\Z .
\]
By \eqref{eq:twist} the two one-sided values at each junction point $2\pi k$ agree, so $w\in W^{1,p}_{\mathrm{loc}}(\R)\cap C(\R)$, and by construction
\begin{equation}\label{eq:quasi}
w(\varphi+2\pi)=e^{2\pi i\beta}w(\varphi),\; \forall\varphi\in\R,
\quad
\abs{w'(\varphi+2\pi)}=\abs{w'(\varphi)}\ \text{for a.e.\ }\varphi\in\R .
\end{equation}
In particular $\abs{w}$ and $\abs{w'}$ are $2\pi$-periodic, so for every $a\in\R$,
\begin{equation}\label{eq:periodint}
\int_{a}^{a+2\pi}\abs{w'}\ud s=\int_{0}^{2\pi}\abs{w'}\ud s.
\end{equation}

By \eqref{eq:quasi}, for every $\varphi\in\R$,
\begin{equation}\label{eq:osc}
\abs{\,w(\varphi+2\pi)-w(\varphi)\,}
=\abs{e^{2\pi i\beta}-1}\,\abs{w(\varphi)} 
=2\sin(\pi\theta)\,\abs{w(\varphi)},
\end{equation}
where we used $\abs{e^{i\alpha}-1}=2\abs{\sin(\alpha/2)}$ and the identity $\abs{\sin(\pi\beta)}=\sin\big(\pi\dist(\beta,\Z)\big)$, valid because $\abs{\sin(\pi\,\cdot\,)}$ is $1$-periodic and even, and $\sin(\pi\,\cdot\,)$ is increasing on $[0,\tfrac12]$.

Since $w\in W^{1,p}_{\mathrm{loc}}(\R)$ is locally absolutely continuous, for every $\varphi\in\R$,
$$
\abs{w(\varphi+2\pi)-w(\varphi)}
=\Big|\int_{\varphi}^{\varphi+2\pi}w'(s)\ud s\Big|
\le$$$$\int_{\varphi}^{\varphi+2\pi}\abs{w'(s)}\ud s
\overset{\eqref{eq:periodint}}{=}\int_{0}^{2\pi}\abs{w'(s)}\ud s
\le(2\pi)^{1-\frac1p}\Big(\int_{0}^{2\pi}\abs{w'}^{p}\ud s\Big)^{\!\frac1p},
$$
by H\"older's inequality with exponents $\tfrac{p}{p-1}$ and $p$. Combining it with \eqref{eq:osc}, we obtain
\[
\big(2\sin(\pi\theta)\big)^{p}\,\abs{w(\varphi)}^{p}
\;\le\;(2\pi)^{p-1}\int_{0}^{2\pi}\abs{w'}^{p}\ud s
\quad\text{for all }\varphi\in[0,2\pi] .
\]
Integrating over $\varphi\in(0,2\pi)$, we get
\[
\big(2\sin(\pi\theta)\big)^{p}\int_{0}^{2\pi}\abs{w}^{p}\ud\varphi
\;\le\;(2\pi)^{p}\int_{0}^{2\pi}\abs{w'}^{p}\ud\varphi,\]
that is, \[
\frac{\int_{0}^{2\pi}\abs{w'}^{p}}{\int_{0}^{2\pi}\abs{w}^{p}}
\;\ge\;\Big(\frac{\sin(\pi\theta)}{\pi}\Big)^{p}.
\]
Taking the infimum over admissible $w$ in \eqref{eq:lambdatwist} proves the lower bound in \eqref{eq:twosided}. Moreover, since $\sin(\pi\theta)\ge 2\theta$ for $\theta\in[0,\tfrac12]$, we arrive at the lower bound in \eqref{eq:comparable}. 

Now we provide the upper bound.
Test \eqref{eq:lambdadef} with the periodic exponentials $u_n(\varphi):=e^{in\varphi}$, $n\in\Z$: then $\partial_\varphi u_n+i\beta u_n=i(n+\beta)u_n$ and $\abs{u_n}\equiv1$, so the Rayleigh quotient equals $\abs{n+\beta}^{p}$. Minimising over $n\in\Z$,
\[
\lambda(\beta,p)\le\min_{n\in\Z}\abs{n+\beta}^{p}=\dist(\beta,\Z)^{p}=\theta^{p}.
\]
This proves the upper bound in \eqref{eq:twosided} and completes the proof of Theorem \ref{thm:angular}. \qed

\section{Proofs of Theorem \ref{thm:main} and Theorem \ref{cor:allp}}\label{sec:reduction}

For the convenience of the reader we present the reduction used in the proof of \cite[Theorem 1.4]{CKLL}.
Lemmas \ref{lem:revmink}, \ref{lem:radial} and \ref{lem:slicing} below are standard.  The only non-explicit ingredient of \cite{CKLL} (the positivity of $\lambda(\beta,p)$) is replaced by Theorem \ref{thm:angular}.

Fix $u\in C_c^{\infty}(\R^{2})$ and pass to polar coordinates $x=(r\cos\varphi,r\sin\varphi)$. Write $u=u(r,\varphi)$ and let $(\hat e_r,\hat e_\varphi)$ denote the moving orthonormal frame. The potential \eqref{eq:AB} is angular 
$A_\beta=-\tfrac{\beta}{r}\,\hat e_\varphi$  and
 $\nabla u=\partial_r u\,\hat e_r+r^{-1}\partial_\varphi u\,\hat e_\varphi$. Hence 
\[
\nabla_{A_\beta}u
=\big(\partial_r u\big)\,\hat e_r+\Big(\frac{\partial_\varphi u-i\beta u}{r}\Big)\,\hat e_\varphi ,
\]
where the sign of $\beta$ is immaterial for all constants. Thus, pointwise on $\R^2\setminus\{0\}$,
\begin{equation}\label{eq:polar}
\abs{\nabla_{A_\beta}u}^{2}
=\abs{\partial_r u}^{2}+\frac{\abs{\partial_\varphi u-i\beta u}^{2}}{r^{2}},
\end{equation}
and therefore
\begin{equation}\label{eq:polarint}
\int_{\R^2}\abs{\nabla_{A_\beta}u}^{p}\ud x
=\int_{0}^{\infty}\!\!\int_{0}^{2\pi}
\Big(\abs{\partial_r u}^{2}+\tfrac{\abs{\partial_\varphi u-i\beta u}^{2}}{r^{2}}\Big)^{\!p/2}
\ud\varphi\, r\ud r .
\end{equation}
We will use the following two marginal quantities/notations 
\[
R(u):=\int_0^\infty\!\!\int_0^{2\pi}\abs{\partial_r u}^{p}\ud\varphi\,r\ud r,
\quad
\Phi(u):=\int_0^\infty\!\!\int_0^{2\pi}\frac{\abs{\partial_\varphi u-i\beta u}^{p}}{r^{p}}\ud\varphi\,r\ud r .
\]

Both are finite for $1<p<2$: $R(u)$ obviously, and $\Phi(u)$ because near the origin $\partial_\varphi u-i\beta u=i\beta u(0)+O(r)$, so the integrand is $O(r^{1-p})$ with $1-p>-1$.

\subsection{Three standard lemmas}

\begin{lemma}\label{lem:revmink}
Let $(\Omega,\mu)$ be a measure space, $0<q<1$, and let $f,g\colon\Omega\to[0,\infty]$ 
be measurable with $0<\norm{f+g}_{q}<\infty$. Then
\begin{equation}\label{eq:revmink}
\norm{f+g}_{q}\ge\norm{f}_{q}+\norm{g}_{q}.
\end{equation}
\end{lemma}

Lemma \ref{lem:revmink} is the case $0<q<1$ of Minkowski's inequality for integrals, in which the inequality sign is reversed. It is precisely the inequality invoked in the proof of \cite[Theorem 1.4]{CKLL}.

\begin{lemma}\label{lem:radial}
For $1<p<2$ with $u\in C_c^{\infty}(\R^{2})$, and for $2\leq p<\infty$ with $u\in C_c^{\infty}(\R^{2}\setminus\{0\})$, we have 
\[
R(u)=\int_{\R^{2}}\Big|\frac{x}{\abs x}\cdot\nabla u\Big|^{p}\ud x
\;\ge\;\Big(\frac{|2-p|}{p}\Big)^{p}\int_{\R^{2}}\frac{\abs{u}^{p}}{\abs{x}^{p}}\ud x .
\]
\end{lemma}

Lemma \ref{lem:radial} is the classical one-dimensional weighted Hardy inequality applied on each ray, as in \cite{CKLL}. Or it can be viewed as a particular case of a more general radial setting (see e.g. \cite[Chapter 2]{RS}). 

\begin{lemma}\label{lem:slicing}
For $1<p<\infty$, $\beta\in\R$ and $u\in C_c^\infty(\R^2\setminus\{0\})$, we have 
\[
\Phi(u)\;\ge\;\lambda(\beta,p)\int_{\R^{2}}\frac{\abs u^{p}}{\abs x^{p}}\ud x .
\]
\end{lemma}

\begin{proof}
For every fixed $r>0$ the function $\varphi\mapsto u(r,\varphi)$ is smooth and $2\pi$-periodic, hence admissible in \eqref{eq:lambdadef}, so
\[
\int_{0}^{2\pi}\abs{\partial_\varphi u(r,\varphi)-i\beta u(r,\varphi)}^{p}\ud\varphi
\ge\lambda(\beta,p)\int_{0}^{2\pi}\abs{u(r,\varphi)}^{p}\ud\varphi .
\]
Multiplying by $r^{1-p}$ and integrating over $r\in(0,\infty)$ gives the claim.
\end{proof}

\subsection{Proof of Theorem \ref{thm:main}}
Since $1<p<2$ we have $0<p/2<1$. Apply Lemma \ref{lem:revmink} on the measure space $\big((0,\infty)\times(0,2\pi),\,r\ud r\ud\varphi\big)$ with $q=p/2$ to the nonnegative functions $f=\abs{\partial_r u}^{2}$ and $g=r^{-2}\abs{\partial_\varphi u+i\beta u}^{2}$; by \eqref{eq:polarint},
\[
\Big(\int_{\R^2}\abs{\nabla_{A_\beta}u}^{p}\ud x\Big)^{2/p}
=\norm{f+g}_{p/2}
\ge\norm{f}_{p/2}+\norm{g}_{p/2}
=R(u)^{2/p}+\Phi(u)^{2/p},
\]
that is,
\begin{equation}\label{eq:afterminkowski}
\int_{\R^2}\abs{\nabla_{A_\beta}u}^{p}\ud x\ge
\Big[R(u)^{2/p}+\Phi(u)^{2/p}\Big]^{p/2}.
\end{equation}
Inserting Lemmas \ref{lem:radial} and \ref{lem:slicing} and using that $t\mapsto t^{2/p}$ is increasing, we obtain
\begin{equation}\label{eq:reduced}
\int_{\R^2}\abs{\nabla_{A_\beta}u}^{p}\ud x
\ge
\Big[\Big(\frac{2-p}{p}\Big)^{2}+\lambda(\beta,p)^{2/p}\Big]^{p/2}
\int_{\R^{2}}\frac{\abs u^{p}}{\abs x^{p}}\ud x .
\end{equation}
This is exactly the point reached in the proof of \cite[Theorem 1.4]{CKLL}, where the argument concludes with the qualitative statement $\lambda(\beta,p)>0$. By the lower bound of Theorem \ref{thm:angular}, $\lambda(\beta,p)^{2/p}\ge\big(\sin(\pi\theta)/\pi\big)^{2}$, inequality \eqref{eq:main} follows. 
 \qed

\subsection{Proof of Theorem \ref{cor:allp}}
Let $2\leq p<\infty$ and $u\in C_c^\infty(\R^2\setminus\{0\})$; then $\Phi(u)<\infty$ because $u$ vanishes near the origin. By \eqref{eq:polar}, pointwise
$\abs{\nabla_{A_\beta}u}\ge r^{-1}\abs{\partial_\varphi u-i\beta u}$,
whence applying both Lemma \ref{lem:slicing} and Lemma \ref{lem:radial}, we obtain, for $u\in C_c^\infty(\R^2\setminus\{0\})$,
$$
\int_{\R^2}|\nabla_{A_\beta}u|^p\,dx
\ge R(u)+\Phi(u)
\ge
\left[
\left(\frac{p-2}{p}\right)^p+
\lambda(\beta,p)
\right]
\int_{\R^2}\frac{|u|^p}{|x|^p}\,dx
$$$$\ge
\left[
\left(\frac{p-2}{p}\right)^p+
\left(\frac{\sin(\pi\theta)}{\pi}\right)^p
\right]
\int_{\R^2}\frac{|u|^p}{|x|^p}\,dx.
$$
Here we have used the fact 
$(X^2+Y^2)^{p/2}\ge X^p+Y^p,$
for $p\ge2.$
\qed

\section{The complex AB potential}
\label{sec:complexAB}

In this section, we apply the compactness-free approach  (of Theorem~\ref{thm:main}) 
to the complex AB potential,
introduced by Krej\v{c}i\v{r}\'ik \cite{Kre19}. In the linear case
$p=2$ treated in \cite[Theorem~2]{Kre19}, the lower bound $c_\infty\ge\lambda_\alpha$
furnished by its proof is known, with $\lambda_\alpha=\dist(\alpha,\Z)^2$,
because the momentum operator on the circle is normal and diagonalises in the
Fourier basis. Therefore,  the point of interest is the nonlinear regime
$p\neq2$, where diagonalisation is unavailable. Our obtained $L^{p}-$Hardy inequalities for the complex AB potential appear to be new.

Following \cite[eq.~(5)]{Kre19}, let $\alpha\in\C$ and consider the complex AB
potential
\begin{equation}\label{eq:cAB}
   \ab(x) \;=\; (-x_2,x_1)\,\frac{\alpha}{|x|^2},
   \quad \alpha=\mu+i\nu,\; \mu:=\Real\alpha,\ \nu:=\Imag\alpha,
\end{equation}
acting through the magnetic gradient $\nabla_{\ab}u:=\nabla u-i\ab u$ on
complex-valued functions. One has $\ab=\alpha\,\hat e_\phi/r$ in polar coordinates
$x=(r\cos\phi,r\sin\phi)$, whence $$\nabla_{\ab}u=\partial_r u\,\hat e_r+
r^{-1}(\partial_\phi u-i\alpha u)\hat e_\phi.$$ Since $\hat e_r$ and $\hat e_\phi$ are
real orthonormal vectors, taking moduli gives, pointwise on
$\R^2\setminus\{0\}$,
\begin{equation}\label{eq:split}
   |\nabla_{\ab}u|^2
   \;=\; |\partial_r u|^2+\frac{|\partial_\phi u-i\alpha u|^2}{r^2},
\end{equation}
in analogy with \eqref{eq:polar}. The complex value of $\alpha$ is immaterial for
\eqref{eq:split}. Thus, the whole polar reduction of Section~\ref{sec:reduction} is available
verbatim, provided the real angular constant $\lambda(\beta,p)$ of \eqref{eq:lambdadef} is
replaced by its complex counterpart
\begin{equation}\label{eq:lambdaC}
   \lambda(\alpha,p)
   :=
   \inf_{\substack{u\in W^{1,p}(0,2\pi)\setminus\{0\}\\ u(0)=u(2\pi)}}
   \frac{\displaystyle\int_0^{2\pi}\bigl|\partial_\phi u-i\alpha u\bigr|^p\,d\phi}
        {\displaystyle\int_0^{2\pi}|u|^p\,d\phi},
   \quad \alpha\in\C.
\end{equation}
For real $\alpha=-\beta$ this reduces to $\lambda(\beta,p)$ of \eqref{eq:lambdadef}. Throughout,
$\dist(\alpha,\Z)$ denotes the distance in $\C$, so that
$\dist(\alpha,\Z)^2=\dist(\mu,\Z)^2+\nu^2$.

Now we are in the position to state the complex analogue of Theorem~\ref{thm:angular}. 
\begin{theorem}\label{thm:C}
Let $1<p<\infty$ and $\alpha=\mu+i\nu\in\C\setminus\Z$. Then
\begin{equation}\label{eq:mainC}
   \left(\frac{|\nu|}{\sinh(\pi|\nu|)}\right)^{\!p}
   \bigl(\sin^2(\pi\mu)+\sinh^2(\pi\nu)\bigr)^{p/2}
   \le \lambda(\alpha,p)\le
 \dist(\alpha,\Z)^p .
\end{equation}
\end{theorem}

\begin{proof}
Since $\alpha\notin\Z$, the operator
$D_\alpha:=\partial_\phi-i\alpha$ with periodic boundary conditions is a
bijection from $\{u\in W^{1,p}(0,2\pi):u(0)=u(2\pi)\}$ onto $L^p(0,2\pi)$. Set
$T_\alpha:=D_\alpha^{-1}$. Because $u\mapsto g:=D_\alpha u$ is a bijection
between the two admissible classes, \eqref{eq:lambdaC} can be rewritten as
\begin{equation}\label{eq:normrep}
   \lambda(\alpha,p)
   =\inf_{g\neq0}\frac{\|g\|_p^p}{\|T_\alpha g\|_p^p}
   =\|T_\alpha\|_{L^p\to L^p}^{-p}.
\end{equation}
Solving $u'-i\alpha u=g$ with $u(0)=u(2\pi)$ gives
$$u(\phi)=e^{i\alpha\phi}\bigl(u(0)+\int_0^\phi e^{-i\alpha s}g(s)\,ds\bigr)$$ and
$$u(0)=-(1-e^{-2\pi i\alpha})^{-1}\int_0^{2\pi}e^{-i\alpha s}g(s)\,ds.$$ Hence
$T_\alpha$ is the integral operator $u(\phi)=\int_0^{2\pi}G_\alpha(\phi,s)\,g(s)\,ds$
with the periodic Green kernel
\begin{equation}\label{eq:green}
   G_\alpha(\phi,s)
   = e^{\,i\alpha(\phi-s)}
     \left(\mathbf 1_{\{s<\phi\}}-\frac{1}{1-e^{-2\pi i\alpha}}\right),
   \quad \phi,s\in(0,2\pi).
\end{equation}
We estimate $\|T_\alpha\|_{L^p\to L^p}$ by Schur's test. Using
$|e^{i\alpha t}|=e^{-\nu t}$ and the identity
$|1-e^{-2\pi i\alpha}|=e^{2\pi\nu}\,|e^{2\pi i\alpha}-1|$, the kernel modulus is
\[
   |G_\alpha(\phi,s)|=
   \begin{cases}
     \dfrac{e^{-\nu(\phi-s)}}{|e^{2\pi i\alpha}-1|}, & s<\phi,\\[2ex]
     \dfrac{e^{-\nu(\phi-s)}}{e^{2\pi\nu}\,|e^{2\pi i\alpha}-1|}, & s>\phi.
   \end{cases}
\]
Splitting the $s$-integral at $s=\phi$ and computing the two elementary
integrals, the $\phi$-dependence cancels and one obtains, for every
$\phi\in(0,2\pi)$,
\begin{equation}\label{eq:schur}
   \int_0^{2\pi}|G_\alpha(\phi,s)|\,ds
   =\frac{1}{\nu\,|e^{2\pi i\alpha}-1|}
     \Bigl(1-e^{-\nu\phi}+e^{-\nu\phi}-e^{-2\pi\nu}\Bigr)
   =\frac{1-e^{-2\pi\nu}}{\nu\,|e^{2\pi i\alpha}-1|}.
\end{equation}
Set
$$
M(\alpha)=:\frac{1-e^{-2\pi\nu}}{\nu\,|e^{2\pi i\alpha}-1|}.
$$
A similar computation gives $\int_0^{2\pi}|G_\alpha(\phi,s)|\,d\phi=M(\alpha)$
for every $s$. By Schur's test, $\|T_\alpha\|_{L^p\to L^p}\le M(\alpha)$ for all
$1\le p\le\infty$. Finally, from
$|e^{2\pi i\alpha}-1|^2=e^{-2\pi\nu}\bigl(2\cosh2\pi\nu-2\cos2\pi\mu\bigr)
=4e^{-2\pi\nu}\bigl(\sin^2\pi\mu+\sinh^2\pi\nu\bigr)$ and
$1-e^{-2\pi\nu}=2e^{-\pi\nu}\sinh\pi\nu$ we obtain
\begin{equation}\label{eq:M}
   M(\alpha)=\frac{\sinh(\pi|\nu|)}{|\nu|\,\sqrt{\sin^2(\pi\mu)+\sinh^2(\pi\nu)}}.
\end{equation}
Combining \eqref{eq:normrep} with $\lambda(\alpha,p)\ge M(\alpha)^{-p}$ yields the
lower bound in \eqref{eq:mainC}.

To show the upper bound, we test \eqref{eq:lambdaC} with the exponentials
$u_m(\phi):=e^{im\phi}$, $m\in\Z$. Then
$\partial_\phi u_m-i\alpha u_m=i(m-\alpha)u_m$ and $|u_m|\equiv1$, so the
Rayleigh quotient equals $|m-\alpha|^p$. Minimising over $m\in\Z$ yields
$\lambda(\alpha,p)\le\min_{m}|m-\alpha|^p=\dist(\alpha,\Z)^p$.
\end{proof}

Feeding Theorem~\ref{thm:C} into the polar reduction of Section~\ref{sec:reduction} (which, by
\eqref{eq:split}, applies to the complex potential \eqref{eq:cAB} unchanged)
gives explicit $L^{p}$-Hardy constants/inequalities for the complex AB potential.

\begin{corollary}\label{cor:improved}
Let $1<p<2$, $\alpha=\mu+i\nu\in\C\setminus\Z$. Then for all
$u\in C_c^\infty(\R^2)$,
$$
   \int_{\R^2}|\nabla_{\ab}u|^p\,dx \ge
$$$$
   \left[\Bigl(\frac{2-p}{p}\Bigr)^{2}
   +\Bigl(\frac{|\nu|}{\sinh\pi|\nu|}\Bigr)^{2}
   \bigl(\sin^2\pi\mu+\sinh^2\pi\nu\bigr)\right]^{p/2}
   \int_{\R^2}\frac{|u|^p}{|x|^p}\,dx .$$
\end{corollary}

\begin{proof}
As in the proof of Theorem~\ref{thm:main}, we apply the reverse Minkowski inequality
(Lemma~\ref{lem:revmink}) with $q=p/2$ to $f=|\partial_r u|^2$ and $g=r^{-2}|\partial_\phi
u-i\alpha u|^2$ in \eqref{eq:split}, then insert Lemma~\ref{lem:radial} for the radial part
$R(u)$ and similarly Lemma~\ref{lem:slicing} with $\lambda(\alpha,p)$
for the angular part. Near the origin $\partial_\phi u-i\alpha u=-i\alpha
u(0)+O(r)$, so the angular integrand is $O(r^{1-p})$ with $1-p>-1$ and all
quantities are finite for $1<p<2$. And the proof follows from Theorem~\ref{thm:C}. \end{proof}

\begin{corollary}\label{cor:hom}
Let $2\leq p<\infty$, $\alpha=\mu+i\nu\in\C\setminus\Z$. Then for all
$u\in C_c^\infty(\R^2\setminus\{0\})$,
$$
   \int_{\R^2}|\nabla_{\ab}u|^p\,dx
   \ge $$$$
   \left[\Bigl(\frac{p-2}{p}\Bigr)^{p}+\left(\frac{|\nu|}{\sinh(\pi|\nu|)}\right)^{\!p}
   \bigl(\sin^2\pi\mu+\sinh^2\pi\nu\bigr)^{p/2}\right]
   \int_{\R^2}\frac{|u|^p}{|x|^p}\,dx.
$$
\end{corollary}

\begin{proof}
We apply the lower bound from Theorem~\ref{thm:C} to complete the proof as in the proof of Theorem \ref{cor:allp}. 
\end{proof}

\subsection{Three final remarks}

\begin{remark}
As $\nu\to0$ one has $\tfrac{|\nu|}{\sinh\pi|\nu|}\to\tfrac1\pi$ and
$\sqrt{\sin^2\pi\mu+\sinh^2\pi\nu}\to|\sin\pi\mu|=\sin(\pi\theta)$ with
$\theta=\dist(\mu,\Z)$. Thus the lower bound of \eqref{eq:mainC} tends to
$(\sin\pi\theta/\pi)^p$ and Corollaries~\ref{cor:improved}--\ref{cor:hom} reduce
to Theorem~\ref{thm:main} and Theorem~\ref{cor:allp}, respectively. 
\end{remark}

\begin{remark}
For $p=2$ the operator $T_\alpha$ is normal (its kernel \eqref{eq:green}
diagonalises in the Fourier basis, with singular values $|m-\alpha|^{-1}$), so
\eqref{eq:normrep} gives 
$\lambda(\alpha,2)=\dist(\alpha,\Z)^2=\dist(\mu,\Z)^2+\nu^2$. That is, the upper bound in
\eqref{eq:mainC} is attained (cf. the proof of
\cite[Theorem~2]{Kre19}).

For $p\neq2$ the
lower bound in \eqref{eq:mainC} loses only the universal factor already present
in the real case. It remains open  to find explicit sharp Hardy constants for the AB potentials for $p\neq2$.
\end{remark}

\begin{remark}
If $\mu=\Real\alpha\in\Z$ then $\sin\pi\mu=0$ and the two bounds in
\eqref{eq:mainC} coincide,
\[
   \lambda(\alpha,p)=|\Imag\alpha|^p \quad\text{for every }1<p<\infty.
\]
Hence a 
imaginary flux (modulo $\Z$) alone produces the $L^p$-Hardy inequality
\[
   \int_{\R^2}|\nabla_{\ab}u|^p\,dx\;\ge\;|\Imag\alpha|^p
   \int_{\R^2}\frac{|u|^p}{|x|^p}\,dx,
   \quad u\in C_c^\infty(\R^2\setminus\{0\}),
\]
 for all $1<p<\infty$. This is the nonlinear
continuum counterpart of Krej\v{c}i\v{r}\'ik's observation that the imaginary
part of $\alpha$ restores positivity even when its real part is an integer. All constants above are $1$-periodic in $\Real\alpha$ (the gauge
$u\mapsto e^{i\phi}u$ gives $\lambda(\alpha+1,p)=\lambda(\alpha,p)$, i.e.\ flux
quantisation).
\end{remark}

\subsection*{Acknowledgements}
The author thanks Ari Laptev and David Krej\v{c}i\v{r}\'{\i}k for their useful and encouraging comments on this paper. 

\subsection*{Funding}
This research is funded by Nazarbayev University under the grant 110326CRP0806 (D.S.).


\begin{thebibliography}{99}

\bibitem{BDELL} D.~Bonheure, J.~Dolbeault, M.~J.~Esteban, A.~Laptev and M.~Loss, \emph{Inequalities involving Aharonov--Bohm magnetic potentials in dimensions 2 and 3}, Reviews in Mathematical Physics \textbf{33} (2021), 2150006.
\bibitem{BY} N.~Byers and C.~N.~Yang, \emph{Theoretical considerations concerning quantized magnetic flux in superconducting cylinders}, Phys.\ Rev.\ Lett.\ \textbf{7} (1961), 46--49.

\bibitem{CK} C.~Cazacu and D.~Krej\v{c}i\v{r}\'{\i}k, \emph{The Hardy inequality and the heat equation with magnetic field in any dimension}, Commun. Partial Differ. Equ. \textbf{41} (2016), 1056--1088.

\bibitem{CKLL} C.~Cazacu, D.~Krej\v{c}i\v{r}\'{\i}k, N.~Lam and A.~Laptev, \emph{Hardy inequalities for magnetic $p$-Laplacians}, Nonlinearity \textbf{37} (2024), 035004.

\bibitem{CT} X.-P.~Chen and C.-L.~Tang, \emph{Remainder terms of $L^p$-Hardy inequalities with magnetic fields: the case $1<p<2$}, J. Geom. Anal. \textbf{35} (2025), 384.

\bibitem{FKLV} L.~Fanelli, D.~Krej\v{c}i\v{r}\'{\i}k, A.~Laptev and L.~Vega, \emph{On the improvement of the Hardy inequality due to singular magnetic fields}, Commun. Partial Differ. Equ. \textbf{45} (2020), 1202--1212.

\bibitem{Kre19} D.~Krej\v{c}i\v{r}\'{\i}k, 
\emph{Complex magnetic fields: An improved Hardy--Laptev--Weidl inequality and 
quasi-self-adjointness}, SIAM J. Math. Anal. \textbf{51} (2019), 790--807.



\bibitem{LLPP} C.~Lange, S.~Liu, N.~Peyerimhoff and O.~Post, \emph{Frustration index and Cheeger inequalities for discrete and continuous magnetic Laplacians}, Calc.\ Var.\ Partial Differential Equations \textbf{54} (2015), 4165--4196.

\bibitem{LW} A.~Laptev and T.~Weidl, \emph{Hardy inequalities for magnetic Dirichlet forms}, in: Mathematical Results in Quantum Mechanics (Prague, 1998), Oper.\ Theory Adv.\ Appl.\ \textbf{108}, Birkh\"auser, Basel, 1999, pp.~299--305.


\bibitem{LY} G.~Lu and Q.~Yang, \emph{Trudinger--Moser and Hardy--Trudinger--Moser inequalities for the Aharonov--Bohm magnetic field}, Calc.\ Var.\ Partial Differential Equations \textbf{63} (2024), 95.

\bibitem{RS} M. Ruzhansky  and D. Suragan,  \emph{Hardy Inequalities on Homogeneous Groups}, Progress in Mathematics, vol. 327, Springer Basel, 2019.

\end{thebibliography}
\end{document}